\documentclass[a4paper,12pt]{article}
\usepackage[cp1251]{inputenc}            % Выбор языка и кодировки
\usepackage[english]{babel}
\usepackage[a4paper,headheight=17pt,left=25mm,right=15mm,bottom=25mm,top=30mm]{geometry}
\usepackage[ruled,vlined]{algorithm2e}
\usepackage{amssymb}
\usepackage{amsmath, amsthm}
\usepackage{enumitem}

\theoremstyle{plain}
\newtheorem{thm}{Theorem}
\newtheorem{lem}{Lemma}

\newtheorem{cor}{Corollary}[thm]

\theoremstyle{definition}
\newtheorem{example}{Example}

\newtheorem{defn}{Definition}

\begin{document}

\begin{center}\Large
\textbf{Formations of Finite Groups in Polynomial Time:\\
the $\mathfrak{F}$-Hypercenter}\normalsize

\smallskip
Viachaslau I. Murashka

 \{mvimath@yandex.ru\}

Faculty of Mathematics and Technologies of Programming,
Francisk Skorina Gomel State University, Sovetskaya 104, Gomel,
246028, Belarus\end{center}

 \begin{abstract}
  For a wide family of formations $\mathfrak{F}$  (which includes Baer-local formations) it is proved that the $ \mathfrak{F}$-hypercenter
  of a permutation finite group can be computed in polynomial time. In particular, the algorithms for computing the $\mathfrak{F}$-hypercenter for the following classes of
  groups are suggested: hereditary local formations with the Shemetkov property, rank formations, formations of all  quasinilpotent, Sylow tower, $p$-nilpotent, supersoluble, $w$-supersoluble and
  $SC$-groups. For some of these formations algorithms for the computation of the intersection of all maximal $\mathfrak{F}$-subgroups are suggested.
 \end{abstract}


\textbf{AMS}(2010). 20D10, 20B40.

\section*{Introduction}

All groups considered here are finite. Recall that a class of groups is a collection $\mathfrak{X}$ of groups with the property that
if $G\in\mathfrak{X} $ and if $H \simeq  G$, then $H \in \mathfrak{X}$. The theory of classes of groups is well developed nowadays (for example, see \cite{s9, Doerk1992, Guo2015, s6}) and has various applications (for example, in the theory of formal languages \cite{BallesterBolinches2014},  in the solution of Yang-Baxter equation \cite{BallesterBolinches2023} and etc.) The main its tasks   are to construct classes of groups and  to recognize is a  given group belongs to a given class or not. With a class of groups $\mathfrak{X}$ one can associate the canonical subgroups such as the $\mathfrak{X}$-residual, the $\mathfrak{X}$-radical, a $\mathfrak{X}$-projector, a $\mathfrak{X}$-injector, the $\mathfrak{X}$-hypercenter and etc. Even if a group does not belong to $\mathfrak{X}$, then these subgroups encode some of its properties associated with~$\mathfrak{X}$.

The algorithms for computing the $\mathfrak{X}$-residual and $\mathfrak{X}$-projectors (of a soluble group) were presented in \cite{EICK2002, HOFLING2001}. The algorithms for computing the $\mathfrak{X}$-radical and $\mathfrak{X}$-injectors (of a soluble group) were presented in \cite{HOFLING2001}. Nevertheless the algorithms for the computation of the $\mathfrak{X}$-hypercenter which plays an important role in the theory of formations were not suggested before.

A group can be represented in different ways. In this paper we will consider only permutation groups because the computational theory of permutation groups is well developed (see \cite{Seress2003}). We leave  the reader to decide how our algorithms can be applied in the case of non-permutation groups.
The aim of this paper is to find effective algorithms (which runs in polynomial time for permutation groups) for the computation of the  $\mathfrak{F}$-hypercenter for a wide family of formations $\mathfrak{F}$ of not necessary soluble groups.

%Also we can't modify algorithm form Theorem \ref{thm1} to compute  as shown in the following example

\section{The Main Result}

Recall that the hypercenter of a group is just the final member of its upper central series. The concept of the hypercenter was widely studied and  generalized by different mathematicians (Baer \cite{Baer1959}, Huppert \cite{Huppert1969}, Shemetkov \cite{Shemetkov1974}, Skiba \cite{Skiba2012} and many other). Some of their results  are presented in
\cite[Chapter 1]{Guo2015}. In the final form the notion of hypercenter appeared in \cite{s6}.   Let $\mathfrak{X}$ be a class of groups.  A chief factor $H/K$ of  $G$ is called   $\mathfrak{X}$-\emph{central} in $G$ provided    that the semidirect product  $(H/K)\rtimes (G/C_G(H/K))$ of $H/K$ with $G/C_G(H/K)$ corresponding to the action by conjugation of $G$ on
$H/K$ belongs to $\mathfrak{X}$ (see \cite[p. 127--128]{s6}). The  $\mathfrak{X}$-hypercenter $\mathrm{Z}_\mathfrak{X}(G)$ is the greatest normal subgroup of $G$ all whose $G$-composition factors are $\mathfrak{X}$-central.

From Barnes-Kegel Theorem \cite[IV, Proposition 1.5]{Doerk1992} it follows that if $\mathfrak{F}$ is a formation and $G\in\mathfrak{F}$, then $G=\mathrm{Z}_\mathfrak{F}(G)$. The converse of this statement is false as we can see on the example of the class of all abelian groups. Shemetkov \cite{SHEM} asked to described all formations $\mathfrak{F}=(G\mid G=\mathrm{Z}_\mathfrak{F}(G))$ (such formations are called $Z$-saturated \cite{Murashka2022a}). The solution of this problem was started in  \cite{BallesterBolinches1999, Murashka2022a}.   The
method for calculating the $\mathfrak{F}$-hypercenter of a given group
 may be a useful tool in the solution of this problem.

In \cite{Neumann1987} the example of a permutation group of degree $n$ was given such that it has a quotient with no faithful representations of degree less than $2^{n/4}$. That is why the given definition of the $\mathfrak{X}$-central chief factor is not good from the computational point of view. Therefore we will consider a more general definition:

\begin{defn}[{\cite[Definition 2]{Murashka2024}}]\label{def1}
Let $\mathrm{f}$ be a function which assigns 0 or 1 to every group $G$ and its chief factor $H/K$ such that

(1) $\mathrm{f}(H/K, G)=\mathrm{f}(M/N, G)$ whenever $H/K$ and $M/N$ are $G$-isomorphic chief factors of $G$;

(2) $\mathrm{f}(H/K, G)=\mathrm{f}((H/N)/(K/N), G/N)$ for every $N\trianglelefteq G$ with $N\leq K$.\\
Such functions $\mathrm{f}$ will be called chief factor functions.
Denote by $\mathcal{C}(\mathrm{f})$ the class of groups
\begin{center}
  $(G\mid G\simeq 1$ or $\mathrm{f}(H/K, G)=1$ for every chief factor $H/K$ of the group $G)$.
\end{center}
\end{defn}

In particular if $\mathrm{f}(H/K, G)=1$ iff $H/K$ is $\mathfrak{X}$-central, then $\mathrm{f}$ is a chief factor function. So $Z$-saturated formations (and hence local and Baer-local formations) are the particular cases of this construction.
With each chief factor function we can associate the following subgroup:

\begin{defn}
Denote by $\mathrm{Z}(G,\mathrm{f})$   the greatest normal subgroup of $G$ such that $\mathrm{f}(H/K, G)=1$ for every chief factor $H/K$ of $G$ below it.
\end{defn}

The subgroup $\mathrm{Z}(G,\mathrm{f})$ becomes the $\mathfrak{X}$-hypercenter of $G$ with the right choice of $\mathrm{f}$. In particular, if $\mathrm{f}$ checks if $H/K$ is central  (resp. cyclic) in $G$, then  $\mathrm{Z}(G,\mathrm{f})$ is the (resp. supersoluble) hypercenter.

\begin{example}
Let $\mathrm{f}_1$ and $\mathrm{f}_2$ check if $H/K$ is abelian and if $G/C_G(H/K)$ is soluble respectively. Then $\mathcal{C}(\mathrm{f}_1)=\mathcal{C}(\mathrm{f}_2)=\mathfrak{S}$ is the class of all soluble groups. Note that $\mathrm{Z}(G,\mathrm{f}_1)$ is the soluble radical of $G$ and $\mathrm{Z}(G,\mathrm{f}_2)$ is the soluble hypercenter of $G$. If we consider the semidirect product of the alternating group $A_5$ of degree 5 and its faithful simple module over $\mathbb{F}_5$ (it exists by \cite[B, Theorem 10.3]{Doerk1992}), then we will see that these subgroups are different.
\end{example}

The main result of the paper is

\begin{thm}\label{thm1}
Assume that $\mathrm{f}(H/T, G)$ can be computed in polynomial time $($in $n)$ for every group $G\leq S_n$ and its chief factor $H/T$. Then
the  subgroup $\mathrm{Z}(G/K,\mathrm{f})$ is well defined and can be computed in polynomial time $($in $n)$ for every $K\trianglelefteq G\leq S_n$.
\end{thm}

In section \ref{appl} we will discuss how this theorem can be applied to various formations.

\section{Preliminaries}
All unexplained notations and terminologies are standard. The reader is referred to \cite{s9, Doerk1992, Guo2015} if necessary. Recall that
$\mathrm{Z}(G)$ denotes the center of $G$; $\mathrm{O}_p(G)$ is the greatest normal $p$-subgroup of $G$; $\mathrm{O}^\pi(G)$ is the smallest normal subgroup of $G$  of $\pi$-index; $\Omega_1(G)$ denotes the subgroup that is generated by elements of order $p$ for a $p$-group   $G$; $S_n$ denotes the symmetric group of degree~$n$; $\mathfrak{G}_\pi$ and $\mathfrak{S}$ denote the classes of all $\pi$-groups and soluble groups respectively.

Recall that a \emph{formation} is a class of groups $\mathfrak{F}$ which is   closed  under taking epimorphic images (i.e. from $G\in\mathfrak{F}$ and $N\trianglelefteq G$ it follows that $G/N\in\mathfrak{F}$)  and subdirect products (i.e. from $G/N_1\in\mathfrak{F}$ and $G/N_2\in\mathfrak{F}$ it follows that $G/(N_1\cap N_2)\in\mathfrak{F}$). If $\mathfrak{F}$ is a non-empty formation, then in every group $G$ exists the $\mathfrak{F}$-residual, i.e. the smallest normal subgroup of $G$ with $G/G^\mathfrak{F}\in\mathfrak{F}$.

We use standard computational conventions of abstract finite groups equipped
with poly\-nomial-time procedures to compute products and inverses of elements (see \cite[Chapter 2]{Seress2003}).
For both input and output, groups are specified by generators. We will consider only $G=\langle S\rangle\leq S_n$ with $|S|\leq n^2$. If necessary, Sims' algorithm \cite[Parts 4.1 and 4.2]{Seress2003} can be used to arrange that $|S|\leq n^2$. Quotient groups are specified by generators of a group and its normal subgroup.
We need the following well known basic tools in our proofs (see, for example \cite{Kantor1990a} or \cite{Seress2003}).
Note that some of them are obtained mod CFSG.

\begin{thm}\label{Basic}
  Given $A, B\trianglelefteq G\leq S_n$ with $A\leq B$, in polynomial time one can solve the following problems:

  \begin{enumerate}

\item\label{i1}         Find $C_{G/A}(B/A)$. %\cite[P6(i)]{Kantor1990a}.

\item\label{i2} Find $G'$ and hence $G^\mathfrak{S}$.

    \item\label{i3}  Find $|G|$.

\item\label{i4} Find $\mathrm{Z}(G/A)$,
% \cite[P14(iii)]{Kantor1990a},
$\mathrm{O}_p(G/A)$ %\cite[P16(i)]{Kantor1990a}
    and $\mathrm{O}^\pi(G/A)$.
     % \cite[P16(ii)]{Kantor1990a}.

    \item\label{i5} Find a chief series for $G$ containing $A$ and $B$.

\item\label{i6} Given   $H=\langle S_1\rangle, K=\langle S_2\rangle \leq G$ find $\langle H, K\rangle=\langle S_1, S_2\rangle$ and $[H, K]=\langle \{[s_1, s_2]\mid s_1\in S_1, s_2\in S_2\rangle^{\langle H, K\rangle}$.

\item\label{i7} Check  if $G/A$ is simple.

% Given $H, K\leq G$ with $K\trianglelefteq G$ find $H\cap K$.

 %   \item Given $T\subseteq G$ find $\langle T\rangle^G$.

%    \item (mod CFSG) Given a prime $p$ dividing $|G|$, find a Sylow $p$-subgroup $P$ of $G$ \cite{Kantor1990}.

%\item   Chief series of $G$ .
  \end{enumerate}
\end{thm}

The following lemma plays an important role in our proves.

\begin{lem}[\cite{Babai1986}]\label{chain}
   Given $G \leq S_n$ every chain of subgroups of $G$ has at most $2n-3$  members for $n\geq 2$.
\end{lem}

\section{Proof of Theorem \ref{thm1}}

The proof of Theorem \ref{thm1} is based on the following 4 lemmas.

\begin{lem}\label{a}
The following statements hold:

$(1)$  $\mathrm{Z}(G,\mathrm{f})$ is well defined for any group $G$.

$(2)$ Let $Z/K=\mathrm{Z}(G/K,\mathrm{f})$. Then $Z$ is the greatest normal subgroup of $G$ such that it contains $K$ and $\mathrm{f}(H/T, G)=1$ for every chief factor $H/T$ of $G$ with $K\leq T\leq H\leq Z$.
\end{lem}

\begin{proof}
$(1)$  Let $M$ and $N$ be normal subgroups of a group $G$. Then form \cite[A, The Isomorphism Theorems(2)]{Doerk1992} it follows that every chief factor of $G$ below $MN$ is $G$-isomorphic to a chief factor of $G$ below either $M$ or $N$. So if $\mathrm{f}(H/K, G)=1$ for every chief factor $H/K$ of $G$ below $M$ and $N$, then $\mathrm{f}(H/K, G)=1$ for every chief factor $H/K$ of $G$ below $MN$ by $(1)$ of Definition \ref{def1}. It means that the subgroup $\mathrm{Z}(G,\mathrm{f})$ is well defined.

$(2)$ Let $Z$ be the greatest normal subgroup of $G$ such that it contains $K$ and $\mathrm{f}(H/T, G)=1$ for every chief factor $H/T$ of $G$ with $K\leq T\leq H\leq Z$. Now $\mathrm{f}((H/K)/(T/K), G/K)=\mathrm{f}(H/T, G)=1$ for every chief factor $H/T$ of $G$ such that  $(H/K)/(T/K)$ is a chief factor of $G/K$ below $Z/K$ by $(2)$ of Definition \ref{def1}. Hence $Z/K\leq \mathrm{Z}(G/K,\mathrm{f})$. Assume that $\mathrm{f}((M/K)/(Z/K), G/K)=1$ for some chief factor $(M/K)/(Z/K)$ of $G/K$. Then $\mathrm{f}(M/Z, G)=1$ for some chief factor $M/Z$ of $G$ by $(2)$ of Definition \ref{def1}, a contradiction with the definition of $Z$. Thus $Z/K=\mathrm{Z}(G/K,\mathrm{f})$.
\end{proof}

\begin{lem}\label{c}
  Let $H/T$ be a chief factor of $G$, $K\trianglelefteq G$, $K\leq T$, $Z_1/K= \mathrm{Z}(G/K,\mathrm{f})\cap T/K$ and
  $Z_2/K= \mathrm{Z}(G/K,\mathrm{f})\cap H/K$. Then $Z_2/Z_{1}\not\simeq 1$ if and only if $\mathrm{f}(H/T, G)=1$ and $H/Z_{1}=T/Z_{1}\times Z/Z_{1}$ for some $Z\trianglelefteq G$. In this case $Z_2=Z$.
\end{lem}

\begin{proof}
  Note that $Z_{1}/K=Z_2/K\cap T/K=(Z_2\cap T)/K$. Hence $Z_1=Z_2\cap T$. Now
  $$Z_2/Z_{1}=Z_2/(Z_2\cap T)\simeq Z_2T/T\leq H/T.$$
Hence $Z_2/Z_{1}$ is $G$-isomorphic to either 1 or $H/T$.

 Assume that $Z_2/Z_1\not\simeq 1$.  It means that $1=\mathrm{f}(Z_2/Z_1, G)=\mathrm{f}(H/T, G)$ and $Z_2/Z_1$ is a minimal normal subgroup of $H/Z_{1}$ not contained in $T/Z_{1}$. Hence  $H/Z_{1}=T/Z_{1}\times Z_2/Z_{1}$.

 Assume now that $\mathrm{f}(H/T, G)=1$ and $H/Z_{1}=T/Z_{1}\times Z/Z_{1}$ for some $Z\trianglelefteq G$. Note that $\mathrm{f}(A/B, G)=1$ for every chief factor $A/B$ of $G$ between $Z_{1}$ and $K$. From $H/T=ZT/T\simeq Z/(Z\cap T)=Z/Z_{1}$ it follows that $\mathrm{f}(Z/Z_{1}, G)=1$. It means that $$Z_{1}/K<Z/K\leq \mathrm{Z}(G/K,\mathrm{f})\cap H/K=Z_2/K.$$ Therefore $Z_2/Z_{1}\not\simeq 1$ and $Z=Z_2$.
\end{proof}

\begin{lem}\label{d1}
  In the notations of Lemma \ref{c} if $H/T$ is non-abelian and $Z_1$ is given, then we can compute $Z_2$ in polynomial time.
\end{lem}

\begin{proof}
  Assume that $H/T$ is non-abelian. If $\mathrm{f}(H/T, G)=0$, then $Z_1=Z_2$ by Lemma \ref{c}. So we can assume that $\mathrm{f}(H/T, G)=1$. Let $C/Z_1=C_{H/Z_{1}}(T/Z_{1})^\mathfrak{S}$. We claim that $Z_2/Z_1\not\simeq 1$ iff $|C/Z_1|=|H/T|$. In this case if  $Z_2/Z_1\not\simeq 1$, then  $Z_2/Z_1=C/Z_1$.

  Suppose that  $Z_2/Z_1\not\simeq 1$. From Lemma \ref{c} it follows that  $H/Z_{1}=T/Z_{1}\times Z/Z_{1}$ for some $Z\trianglelefteq G$. Then $Z/Z_{1}\leq C_{H/Z_{1}}(T/Z_{1})$. Hence
 \begin{multline*}
  C_{H/Z_{1}}(T/Z_{1})=TZ/Z_{1}\cap C_{H/Z_{1}}(T/Z_{1})=
  (C_{H/Z_{1}}(T/Z_{1})\cap T/Z_{1})(Z/Z_{1})=\mathrm{Z}(T/Z_{1})\times(Z/Z_{1}).\end{multline*}
  Recall that $H/T\simeq Z/Z_1$ is a direct product of simple non-abelian groups. It means that $ C_{H/Z_{1}}(T/Z_{1})^\mathfrak{S}=Z/Z_{1}=Z_2/Z_1$ by Lemma \ref{c}. Thus $|C/Z_1|=|H/T|$.

Suppose that $|C/Z_{1}|=|H/T|$. Note that $C/Z_{1}$ is non-abelian (because it is the soluble residual)  and $C/Z_1\cap T/Z_1\leq \mathrm{Z}(T/Z_1)$ is abelian. Therefore $C/Z_{1}\not\leq T/Z_{1}$.  From
$$C/Z_1\textrm{ char }C_{H/Z_{1}}(T/Z_{1})=C_{G/Z_{1}}(T/Z_{1})\cap H/Z_1\trianglelefteq G/Z_1$$
it follows that $C/Z_1\trianglelefteq G$.
It means that $H=TC$. Note that $H/T=CT/T\simeq C/(C\cap T)$. From  $|C/Z_{1}|=|H/T|$   it follows that $|C\cap T|=|Z_{1}|$. Note that $Z_1\leq C\cap T$. Thus  $C\cap T=Z_{1}$. It means that $H/Z_{1}=T/Z_{1}\times C/Z_{1}$. Thus $C/Z_1=Z_2/Z_1$ by Lemma \ref{c}.

By the statement of Theorem \ref{thm1} we can compute $\mathrm{f}(H/T, G)$ in polynomial time. Now $C/Z_1=C_{H/Z_{1}}(T/Z_{1})$ can be computed in polynomial time by \ref{i1} of Theorem \ref{Basic}. Note that $(C/Z_1)^\mathfrak{S}=C^\mathfrak{S}Z_1/Z_1$. Now we can compute $C^\mathfrak{S}$ by \ref{i2} of Theorem \ref{Basic}. According to \ref{i3} of Theorem \ref{Basic}  we can check if $|H/T|=|C/Z_1|$ in polynomial time.
\end{proof}

\begin{lem}\label{d2}
  In the notations of Lemma \ref{c} if $H/T$ is abelian and $Z_1$ is given, then we can compute $Z_2$ in polynomial time.
\end{lem}

\begin{proof}
  Assume that $H/T$ is abelian. If $\mathrm{f}(H/T, G)=0$, then $Z_1=Z_2$ by Lemma \ref{c}. So we can assume that $\mathrm{f}(H/T, G)=1$.

 Note that $H/T$   is a $p$-group for some prime $p$ and every minimal normal $p$-subgroup of a group is centralized by its the greatest normal $p$-subgroup \cite[A, Lemma 13.6]{Doerk1992}.  Also note that every minimal normal $p$-subgroup is generated by elements of order $p$. Hence all minimal normal $p$-subgroups of $H/Z_{1}$ are subgroups of $P/Z_{1}=\Omega_1(\mathrm{Z}(\mathrm{O}_p(H/Z_{1})))$.

 If $P\leq T$, then $H/Z_{1}\neq T/Z_{1}\times Z/Z_{1}$ for all $Z\trianglelefteq G$. Thus $Z_2=Z_1$ by Lemma \ref{c}.
 Assume now that $P\not\leq T$. Note that $P/Z_{1}$ is the direct product of $G$-indecomposable subgroups $P_1/Z_1,\dots, P_m/Z_1$ by the Krull-Remak-Schmidt Theorem. Note that $H/Z_{1}=T/Z_{1}\times Z/Z_{1}$ for some $Z\trianglelefteq G$ iff $P/Z_1=(P\cap T)/Z_1\times Z/Z_1$ for some $G$-simple subgroup $Z/Z_1$ of $P/Z_1$.
 Let\linebreak $(P\cap T)/Z_1$ be the direct product of $G$-indecomposable subgroups $T_1/Z_1,\dots, T_l/Z_1$. If the required $Z$ exists, then $P/Z_1$ is the direct product of $G$-indecomposable subgroups $T_1/Z_1,\dots, T_l/Z_1$ and $Z/Z_1$. Hence $l=m-1$. Moreover if $P/Z_{1}$ is the direct product of $G$-indecomposable subgroups $Q_1/Z_1,\dots, Q_m/Z_1$, then these subgroups can be numbered in such way that $P/Z_{1}$ is the direct product of $T_1/Z_1,\dots, T_{m-1}/Z_1$ and $Q_m/Z_1$ by the Krull-Remak-Schmidt Theorem. Note that in these case $Q_m/Z_1$ is $G$-simple and not contained in $T/Z_{1}$.
  Therefore the required subgroup $Z$ exists only in case when any direct decomposition of $P/Z_1$ into the product of $G$-indecomposable subgroups has $G$-simple subgroup not contained in $T/Z_1$.

  Note that $P/Z_{1}$ can be viewed as a $G$-module over $\mathbb{F}_p$ and   $G$-indecomposable and $G$-simple subgroups of $P/Z_{1}\leq G/Z_1$ are in one to one correspondence of indecomposable and simple submodules of   a $G$-module $P/Z_1$.

 From \ref{i4} of Theorem \ref{Basic} it follows that we can compute $Z(\mathrm{O}_p(H/Z_{1}))$ in polynomial time. Since $Z(\mathrm{O}_p(H/Z_{1}))$ is abelian, we can find $P/Z_1$ in polynomial time by taking the right powers of generators of $Z(\mathrm{O}_p(H/Z_{1}))$.
 According to \cite[p. 155]{Seress2003} for each generator $ g$ of $G$ in polynomial time we can find the linear transformation which this element induces on $P/Z_1$.  Let  $R$ be the algebra  generated by the above mentioned transformations. Now by \cite[Theorem 5]{Chistov1997PolynomialTA}  we can decompose $P/Z_1$  into the sum $P_1\oplus\dots\oplus P_k$  of indecomposable submodules in polynomial time.
 For each $P_i$ we can check either it is simple or not by \cite[Corollary 5.4]{Ronyai1990} and  for each simple $P_i$ we can find the subgroup $P_i/Z_1$ of $P/Z$ to which it corresponds. If $P_i/Z\not\leq T/Z$, then $P_i/Z$ is the required subgroup.
\end{proof}
%From the other hand   Now we are left to check if $P_i/Z_1\leq T/Z_1$.

\subsection{Proof of the Theorem}

We can  compute a chief series $K=G_0\trianglelefteq G_1\trianglelefteq\dots\trianglelefteq G_m=G$ of $G$  in polynomial time by \ref{i5} of Theorem \ref{Basic}.  Let $Z_i=G_i\cap \mathrm{Z}(G/K,\mathrm{f})$. We see that $Z_{i-1}=Z_i\cap G_{i-1}$ and $Z_0=K$.  Now using Lemmas \ref{d1} and \ref{d2} if we know $Z_{i-1}$ then we can compute $Z_i$ in polynomial time. Note that $m\leq 2n$ by Lemma \ref{chain}. Hence we can compute $\mathrm{Z}(G/K,\mathrm{f})$ in polynomial time (see Algorithm 1).

%\subsection{Proof of Corollary \ref{cor1.1}}

\begin{algorithm}[H]
\caption{GFHYPERCENTER$(G, K, \mathrm{f})$}
\SetAlgoLined
\KwResult{$\mathrm{Z}(G/K,\mathrm{f})$.}
\KwData{A normal subgroup $K$ of a  group $G$ and a chief factor function $\mathrm{f}$.}

$Z\gets K$\;
Compute a chief series $K=G_0\trianglelefteq G_1\trianglelefteq\dots\trianglelefteq G_m=G$ of $G$\;
\For{$i\in\{1,\dots, m\}$}
   {\If{$\mathrm{f}(G_i/G_{i-1}, G)=1$}
     {\eIf{$|\pi(G_i/G_{i-1})|>1$}
           {$C/Z\gets C_{G_i/Z}(G_{i-1}/Z)$\;
             $S/Z\gets \langle C^{\mathfrak{S}}, Z\rangle/Z$\;
             \If{$|S/Z|=|G_i/G_{i-1}|$}
                {$Z\gets S$\;}
           }
           {$p\gets\pi(G_i/G_{i-1})[1]$\;
             $P/Z\gets  \Omega_1(\mathrm{Z}(\mathrm{O}_p(G_i/Z)))$\;
              For each generator $ g$ of $G$ find the linear transformation which this element induces on $P/Z$\;
              Let  $R$ be the  algebra generated by above mentioned transformations. Decompose the $R$-module $P/Z$  into the sum $P_1\oplus\dots\oplus P_k$  of indecomposable submodules\;
                \For{$j\in\{1,\dots, k\}$}
                       {\If{$P_j$ is a simple module}
                          {Find the corresponding to $P_j$  subgroup $P_j/Z$\;
                           \If{$\langle P_j/Z, G_{i-1}/Z\rangle=G_i/Z$}
                                 {$Z\gets P_j$\;
                                  Leave the current \textbf{for} cycle\;}
                            }
                       }
                }
       }
       }
\Return{$Z/K$}
 \end{algorithm}

\section{Applications}\label{appl}

Recall \cite{Skiba2012} that  $\mathrm{Int}_\mathfrak{F}(G)$   denotes the intersection of all $\mathfrak{F}$-maximal subgroups of $G$.
L.A.~Shemet\-kov posed
the following question on Gomel Algebraic seminar in 1995: «For
what non-empty (normally) hereditary local (Baer-local) formations $\mathfrak{F}$ do
the equality $\mathrm{Int}_\mathfrak{F}(G) = \mathrm{Z}_\mathfrak{F}(G)$ hold for every group $G$?»
The solution to this question in case when $\mathfrak{F}$ is a hereditary saturated formation was obtained in \cite{Skiba2012}. For some class of Baer-local formation the answer to this question was given in \cite{Murashka2018}. Nevertheless this problem is steel open. We will show that Theorem~\ref{thm1} can be used to compute $\mathrm{Int}_\mathfrak{F}(G)$ for some formations $\mathfrak{F}$.

\subsection{Baer local and local formation}

The notion of $\mathfrak{F}$-hypercenter plays an important role in the study for Baer-local and local formations.
Let $f$ be a function which assigns to every simple group $J$ a possibly empty formation $f(J)$. Now extend the domain of $f$. If $G$ is the direct product of simple groups isomorphic to $J$, then  let $f(G)=f(J)$. If $J$ is a cyclic group of order $p$, then let $f(p)=f(J)$. Such functions $f$ are called  Baer functions \cite[IV, Definitions 4.9]{Doerk1992}.  A formation $\mathfrak{F}$ is called Baer-local (or composition, see \cite[p. 4]{Guo2015} and \cite{Shemetkov1974}) if for some Baer function $f$ holds:
$$\mathfrak{F}=(G\mid G/C_G(H/K)\in f(H/K)\textrm{ for every chief factor }H/K\textrm{ of }G).$$

\begin{thm}
  Let $\mathfrak{F}$ be a Baer-local  formation defined by $f$.
  Assume that  $(G/K)^{f(J)}$ can be computed in polynomial time for every  $K\trianglelefteq G\leq S_n$ and a simple group $J$. Then $\mathrm{Z}_\mathfrak{F}(G/K)$ can be computed in polynomial time for every  $K\trianglelefteq G\leq S_n$ and a simple group $J$.
\end{thm}

\begin{proof} Here we assume that if $f(J)=\emptyset$, then the computation of $(G/K)^{f(J)}$ returns ``is not defined''.
  Let $\mathrm{f}(H/K, G)=1$ iff $G/C_G(H/K)\in f(H/K)$. This condition is equivalent to $G^{f(H/K)}$ is defined and $[G^{f(H/K)}, H]\subseteq K$. Hence it can be checked in polynomial time by \ref{i3} and \ref{i6} of Theorem \ref{Basic} and the statement of the theorem. By analogy with the proof of \cite[Theorem 7]{Murashka2024}  one can show that $\mathrm{f}$ is a chief factor function.   Then $\mathfrak{F}=\mathcal{C}(\mathrm{f})$. Hence we can compute $G^\mathfrak{F}$ in polynomial time by \cite[Theorem 1]{Murashka2024}.

Let  $H/K$ be a    chief factor of $G$ and $T\simeq (H/K)\rtimes G/C_G(H/K)$. Note that $T$ is a primitive group of type 1 or 3. Then $T/C_T(H/K)\simeq G/C_G(H/K)$ by \cite[A, Theorem 15.2(1, 3)]{Doerk1992}.  So if $T\in\mathfrak{F}$, then
$G/C_G(H/K)\in f(H/K)\cap \mathfrak{F}$. From the other hand if $G/C_G(H/K)\in f(H/K)\cap \mathfrak{F}$, then $T\in\mathfrak{F}$. Note that $G/C_G(H/K)\in (f(H/K)\cap \mathfrak{F})$ iff $[G^\mathfrak{F}G^{f(H/K)}, H]\subseteq K$.
Let $\mathrm{f}_1(H/K, G)=1$ iff $G^{f(H/K)}$ is defined and $[G^\mathfrak{F}G^{f(H/K)}, H]\subseteq K$. Note that we can check this condition in polynomial time by \ref{i3} and \ref{i6} of Theorem \ref{Basic}, $\mathrm{f}_1$ is a chief factor function and $\mathrm{Z}(G, \mathrm{f}_1)=\mathrm{Z}_\mathfrak{F}(G)$ for any group $G$. Thus $\mathrm{Z}_\mathfrak{F}(G/K)$ can be computed in polynomial time for any $K\trianglelefteq G\leq S_n$ by Theorem~\ref{thm1}.
\end{proof}

A Baer-local formation defined by Baer function $f$  is called local if $f(J)=\cap_{p\in\pi(J)}f(p)$ for every simple group $J$.
In this case to define $\mathfrak{F}$ we need only to know the values of $f$ on primes.

\begin{cor}\label{cor1.1}
  Let $\mathfrak{F}$ be a local  formation locally defined by $f$.
  Assume that $(G/K)^{f(p)}$ can be computed in polynomial time for every  $K\trianglelefteq G\leq S_n$ and a prime $p$. Then $\mathrm{Z}_\mathfrak{F}(G/K)$ can be computed in polynomial time.
\end{cor}

According to \cite[the proof of Corollary 1]{Murashka2024}   for the class $\mathfrak{U}$ of supersoluble groups, the class $w\mathfrak{U}$  of widely supersoluble groups \cite{Vasilev2010}, the class $\mathfrak{N}\mathcal{A}$ of groups $G$ such that all Sylow subgroups of $G/\mathrm{F}(G)$ are abelian \cite{Vasilev2010}, the class $sm\mathfrak{U}$ of groups with submodular Sylow subgroups \cite{Vasilyev2015, Zimmermann1989}, the class of strongly supersoluble groups $s\mathfrak{U}$ \cite{Vasilyev2015} and the class $sh\mathfrak{U}$ of groups all of whose Schmidt subgroups are supersoluble \cite{Monakhov1995} we can  compute the described in the statement of Corollary \ref{cor1.1} $f(p)$-residuals   in polynomial time.

\begin{cor}
  Let  $\mathfrak{F}\in\{\mathfrak{U}, \mathrm{w}\mathfrak{U}, s\mathfrak{U}, sm\mathfrak{U}, \mathfrak{N}\mathcal{A}, sh\mathfrak{U}\}$. Then $\mathrm{Z}_\mathfrak{F}(G/K)$ can be computed in a polynomial time for every $K\trianglelefteq G\leq S_n$.  \end{cor}

  Recall \cite{Vasilev1997} that a group is called $c$-\emph{supersoluble} in the terminology of Vedernikov ($SC$-group in the terminology  of   Robinson \cite{Robinson2001}) if every its chief factor is a simple   group. It was proved the the class $\mathfrak{U}_c$ of all $c$-supersoluble groups is a Baer-local formation \cite{Vasilev1997}.

\begin{thm}
  In polynomial time one can compute $\mathrm{Int}_{\mathfrak{U}_c}(G/K)$ and $\mathrm{Z}_{\mathfrak{U}_c}(G/K)$ for $K\trianglelefteq G\leq S_n$.
\end{thm}

\begin{proof}
  Let $\mathrm{f}_1$ checks if $H/K$ is simple. Then $\mathrm{f}_1(H/K, G)$ can be computed in polynomial time by \ref{i7} of Theorem \ref{Basic}.  Therefore $\mathfrak{U}_c=\mathcal{C}(\mathrm{f}_1)$. It is clear that $\mathrm{Z}(G, \mathrm{f}_1)$ is the greatest normal subgroup of $G$ all whose $G$-composition factors are simple. Therefore $\mathrm{Z}(G, \mathrm{f}_1)H\in\mathfrak{U}_c$ for every $\mathfrak{U}_c$-subgroup $H$ of $G$. Hence $\mathrm{Z}(G, \mathrm{f}_1)=\mathrm{Int}_{\mathfrak{U}_c}(G)$.

  From the other hand let $G\simeq \mathrm{Aut}(P\Omega^+_8(p^f))$ for an odd $p$ and $P\Omega^+_8(p^f)\simeq I\trianglelefteq G$. Note that $I\leq \mathrm{Int}_{\mathfrak{U}_c}(G)$. Now $G/C_G(I)\not\in\mathfrak{U}_c$ because it contains a normal section isomorphic to $S_4$. Note that an abelian chief factor is $\mathfrak{U}_c$-central iff it is simple. If $H/K$ is a non-abelian factor, then $T=(H/K)\rtimes (G/C_G(H/K))$ has two minimal normal subgroups whose quotients are isomorphic to $G/C_G(H/K)$ by \cite[A, Theorem 15.2(3)]{Doerk1992}. Now $T\in \mathfrak{U}_c$ iff $G/C_G(H/K)\in\mathfrak{U}_c$. The last is equivalent to $[G^{\mathfrak{U}_c}, H]\subseteq K$. According to \cite[Theorem 1]{Murashka2024} we can compute $G^{\mathfrak{U}_c}$ in  polynomial time. Therefore we can check if $H/K$ is $\mathfrak{U}_c$-central in $G$ in polynomial time and hence compute $\mathrm{Z}_{\mathfrak{U}_c}(G/K)$.
\end{proof}

Recall that a group is called quasinilpotent if every its element induces an inner automorphism on every its chief factor. The class of all quasinilpotent groups is denoted by $\mathfrak{N}^*$.

\begin{thm}
  In polynomial time one can compute $\mathrm{Int}_{\mathfrak{N}^*}(G/K)=\mathrm{Z}_{\mathfrak{N}^*}(G/K)$ for $K\trianglelefteq G\leq S_n$.
\end{thm}

\begin{proof}
  According to \cite[Corollary 1]{Murashka2018} $\mathrm{Int}_{\mathfrak{N}^*}(G)=\mathrm{Z}_{\mathfrak{N}^*}(G)$ holds for every group  $G$. From \cite[Remark 2]{Murashka2018} $\mathrm{Z}_{\mathfrak{N}^*}(G)$ is the greatest normal subgroup
of $G$ such that every element of $G$ induces an inner automorphism on every chief
factor of $G$ below $\mathrm{Z}_{\mathfrak{N}^*}(G)$, i.e. $\mathrm{Z}_{\mathfrak{N}^*}(G)=\mathrm{Z}(G, \mathrm{f})$ where $\mathrm{f}$ checks if every element of $G$ induces an inner automorphism on $H/K$ that is $G=HC_G(H/K)$ or $G/K=(H/K)C_{G/K}(H/K)$. This condition can be checked in polynomial time by Theorem~\ref{Basic}. It is straightforward to check that $\mathrm{f}$ is a chief factor function. Thus  in polynomial time one can compute $\mathrm{Int}_{\mathfrak{N}^*}(G/K)=\mathrm{Z}_{\mathfrak{N}^*}(G/K)$ for $K\trianglelefteq G\leq S_n$ by Theorem \ref{thm1}.
\end{proof}

\subsection{Formations with the Shemetkov property}

Recall \cite{Semenchuck1984} (see also \cite[Chapter 6]{s9}) that a formation $\mathfrak{F}$ has the Shemetkov property (resp. in $\mathfrak{S}$) if every (resp. soluble) $\mathfrak{F}$-critical group is either a Schmidt group or a cyclic group of prime order. Note that if a hereditary local formation $\mathfrak{F}$ has the Shemetkov property (and contains all nilpotent groups) then it
can be locally defined by $f$ where $f(p)=\mathfrak{G}_{g(p)}$ and $g$ assigns to a prime $p$ a set of primes $g(p)$  with $p\in g(p)$. The converse is not true. Nevertheless a hereditary local formation $\mathfrak{F}$ of soluble groups has the Shemetkov property in $\mathfrak{S}$ and contains all nilpotent groups iff $\mathfrak{F}$ is locally defined by $f$ where $f(p)=\mathfrak{S}_{g(p)}$ and $g$ assigns to a prime $p$ a set of primes $g(p)$  with $p\in g(p)$. With the help of the next result one can compute the $\mathfrak{F}$-hypercenter and the intersection of all maximal $\mathfrak{F}$-subgroups of a (resp. soluble) group $G$ where $\mathfrak{F}$ is a hereditary formation with the Shemetkov property (resp. in $\mathfrak{S}$).

\begin{thm}\label{shem}
  Let $g$ be a function which assigns to a prime $p$ a set of primes $g(p)$  with $p\in g(p)$ and $h(p)=\mathfrak{G}_{g(p)}$. Assume that $g(p)\cap\pi(G)$ can be computed in a polynomial time for every $p\in\pi(G)$. Let $\mathfrak{F}$ be a local formation defined by $h$. Then   $\mathrm{Int}_\mathfrak{F}(G/K)$ and $\mathrm{Z}_\mathfrak{F}(G/K)$ can be computed in a polynomial time for every $K\trianglelefteq G\leq S_n$.
\end{thm}

\begin{proof}
  Let prove that  $H/K\leq \mathrm{Int}_\mathfrak{F}(G/K)$ iff $G/C_G(H/K)\in h(p)$ for every $p\in\pi(H/K)$.

  Assume that $G/C_G(H/K)\in h(p)$ for every $p\in\pi(H/K)$. Let $M/K$ be an $\mathfrak{F}$-maximal subgroup of $G/K$. Let $T=MH$. Note that $T/C_T(U/V)\in h(p)$ for all $p\in\pi(U/V)$ and every chief factor $U/V$ of $T/K$ below $H/K$. From $T/H\in \mathfrak{F}$ it follows that $T/C_T(U/V)\in h(p)$ for all $p\in\pi(U/V)$ and every chief factor $U/V$ of $T/K$ above $H/K$. Thus $T/K\in\mathfrak{F}$. Hence $T/K=M/K$. Therefore $H/K\leq \mathrm{Int}_\mathfrak{F}(G/K)$.

  Assume that $H/K\leq \mathrm{Int}_\mathfrak{F}(G/K)$. It means that $\pi(H/K)\subseteq g(p)$ for all $p\in\pi(H/K)$. Let $M/K$ be an $\mathfrak{F}$-maximal subgroup of $G/K$ and $K=H_0\trianglelefteq H_1\trianglelefteq\dots\trianglelefteq H_m=H$ be a part of chief series of $M/K$.  Then $(M/K)/C_{M/K}(H_i/H_{i-1})\in h(p)$ for all $p\in\pi(H/K)$. Let $C/K=\cap_{i=1}^m C_{M/K}(H_i/H_{i-1})$. Now $(C/K)/C_{M/K}(H/K)$ is a $\pi(H/K)$-group by [A, Corollary 12.4(a)]. Since $h(p)$ is a formation,  $(M/K)/(C/K)\in h(p)$ for all $p\in\pi(H/K)$. Thus $(M/K)/C_{M/K}(H/K)\in \mathfrak{G}_{g(p)}=h(p)$ for all $p\in\pi(H/K)$.
    Since $g(p)\neq\emptyset$ for all prime $p$, we see that every element of $G/K$ belongs to some its $\mathfrak{F}$-maximal subgroup. From $(M/K)/C_{M/K}(H/K)$\linebreak$\simeq MC_{G/K}(H/K)/C_{G/K}(H/K)$ it follows that $G/C_G(H/K)$ is a  $\mathfrak{G}_{g(p)}$-group for every $p\in\pi(H/K)$, i.e. $G/C_G(H/K)\in h(p)$ for every $p\in\pi(H/K)$.

    From \cite[Theorem C(e)]{Skiba2012} it follows that if $I\leq \mathrm{Int}_\mathfrak{F}(G)$, then $\mathrm{Int}_\mathfrak{F}(G/I)=\mathrm{Int}_\mathfrak{F}(G)/I$. Therefore $\mathrm{Int}_\mathfrak{F}(G/K)=\mathrm{Z}(G/K, \mathrm{f})$ where $\mathrm{f}$ checks if $G/C_G(H/K)\in h(p)$ for every $p\in\pi(H/K)$. The last condition is equivalent to $[\mathrm{O}^{g(p)}, H]\subseteq K$ and can be checked in polynomial time by Theorem \ref{Basic}. Thus the statement of theorem follows from Theorem \ref{thm1} and Corollary \ref{cor1.1}.
\end{proof}

Let $\varphi$ be some linear ordering on $\mathbb{P}$. Recall \cite[IV, Examples 3.4(g)]{Doerk1992} that a group $G$ is called a Sylow tower group of type $\varphi$  if  it has normal Hall $\{p_1, \dots, p_t\}$-subgroups for all $1\leq t\leq k$ where $\pi(G)=\{p_1,\dots, p_k\}$ and $p_i>_{\varphi}p_j$ for $i<j$. Let $g(p)=\{q\mid q\leq_{\varphi} p\}$. Note that the class of  all Sylow tower group of type $\varphi$ can be locally defined by $h(p)=\mathfrak{G}_{g(p)}$.

\begin{cor}
  Let $\varphi$ be a linear ordering on $\mathbb{P}$ such that we can check in polynomial time if $p\leq_\varphi q$ for any primes $p, q\leq n$. If $\mathfrak{F}$ is the class of  all Sylow tower group of type $\varphi$, then $\mathrm{Int}_\mathfrak{F}(G/K)$ and $\mathrm{Z}_\mathfrak{F}(G/K)$ can be computed in a polynomial time for every $K\trianglelefteq G\leq S_n$.
\end{cor}

Note that the class of all $p$-nilpotent groups can be locally defined by $h$ where $h(p)=\mathfrak{G}_p$ and $h(q)=\mathfrak{G}$ for $q\neq p$.

\begin{cor}
 Let $p$ be a prime. If $\mathfrak{F}$ is the class of  all $p$-nilpotent groups, then $\mathrm{Int}_\mathfrak{F}(G/K)$ and $\mathrm{Z}_\mathfrak{F}(G/K)$ can be computed in a polynomial time for every $K\trianglelefteq G\leq S_n$.
\end{cor}

\subsection{Rank formations}

If $H/K$ is a chief factor of $G$, then $H/K=H_1/K\times\dots\times H_k/K$ where $H_i$ are isomorphic simple groups. The number $k$ is called the \emph{rank} of $H/K$ in $G$.

\begin{defn}[{\cite[Chapter VII, Definition 2.3]{Doerk1992}}]
   A \emph{rank function} $R$  is a map which associates with each prime $p$
a set $R(p)$ of natural numbers. With each rank function $R$ we associate a class of soluble groups
\begin{center}
  $\mathfrak{F}(R)=(G\in\mathfrak{S}\,|$ $G\simeq 1$ or  for all $p\in \mathbb{P}$ each $p$-chief factor of $G$ has rank in $R(p)$).
\end{center}
\end{defn}
This class is a formation of soluble groups. If  $R(p)=\{1, 2\}$ for all $p\in \mathbb{P}$, then formation $\mathfrak{F}(R)$ is non-local by  \cite[Chapter VII, Theorem 2.18]{Doerk1992} and hence is non-Baer-local.
If $R(p)=\{1\}$ for all $p\in \mathbb{P}$, then $\mathfrak{F}(R)=\mathfrak{U}$.

\begin{thm}\label{Trank}
  Let $R$ be a rank function.
  Assume that one can test if $n\in R(p)$ in  polynomial $($in $n)$ time for every $p\in \mathbb{P}$. Then  $\mathrm{Z}_{\mathfrak{F}(R)}(G/K)$  can be computed in polynomial time for   every $K\trianglelefteq G\leq S_n$.
\end{thm}

\begin{proof}
  From \ref{i3} of Theorem \ref{Basic} we can compute the rank of $H/K$ in polynomial time. Hence we can check if $G/K\in\mathfrak{F}(\mathcal{R})$ for any $K\trianglelefteq G\leq S_n$ in polynomial time by \ref{i3} and \ref{i5} of Theorem \ref{Basic} and the statement of the theorem.
  Note that a chief factor $H/K$ is $\mathfrak{F}(R)$-central iff $H/K$ is a $p$-group for some prime $p$, the rank of $H/K$ is in $R(p)$ and $G/C_G(H/K)\in\mathfrak{F}(R)$. Note that $C_G(H/K)/K=C_{G/K}(H/K)$. Hence we can check if $H/K$ is $\mathfrak{F}(R)$-central in polynomial time by Theorem \ref{Basic}. Thus the statement of Theorem \ref{Trank} follows from Theorem \ref{thm1}.
\end{proof}

\subsection*{Acknowledgments}

This work is supported by BRFFR $\Phi23\textrm{PH}\Phi\textrm{-}237$.\\
I am grateful to A.\,F. Vasil'ev for helpful discussions.
%I wish to thank  the reviewers for their comments and suggestions.

{\small\bibliographystyle{siam}
\bibliography{Alg2}}

\end{document}